%% file: mainc.tex
\documentclass[runningheads]{llncs}
\makeatletter

\usepackage{amsthm}
\usepackage{newtxtext}
\renewcommand{\@Opargbegintheorem}[4]{%
  #4\trivlist\item[\hskip\labelsep{#3#2\@thmcounterend}]}
\makeatother
\usepackage[%
      rm={oldstyle=false,proportional=true},%
      sf={oldstyle=false,proportional=true},%
      tt={oldstyle=false,proportional=true,variable=false},%
      qt=false%
    ]{cfr-lm}
\usepackage{bm}
\usepackage{newtxmath}
\usepackage{xcolor}
\usepackage{graphicx}
\usepackage{wrapfig}
\usepackage{tikz}
\usepackage{amsmath}
\usepackage{algorithm}
\usepackage[noend]{algpseudocode}
\usepackage[title]{appendix}

\newcommand{\algcomment}[1]{%
    \vspace{0.25\baselineskip}%
    \noindent%
    \hrule 
    \vspace{0.25\baselineskip}%
    {\footnotesize #1\par}%
    }
\newcommand*{\QEDA}{\hfill\ensuremath{\blacksquare}}
\newtheorem{observation}[theorem]{Observation}

\newcommand{\MPworker}[2]{}

\newcommand{\ot}{\widetilde{\mathcal{O}}}
\newcommand{\createminheap}[1]{\textsc{Create\_Min\_Heap($#1$)}}
\newcommand{\extractmin}[1]{\textsc{Extract\_Min($#1$)}}
\newcommand{\insertheap}[2]{\textsc{Insert($#1, #2$)}}
\newcommand{\findpordering}[1]{\textsc{Find\_$p$-Ordering($#1$)}}
\newcommand{\merge}[1]{\textsc{Merge($#1$)}}
\newcommand{\pexpincrease}[1]{\textsc{$p$-Exponent\_Increase($#1$)}}
\newcommand{\correlate}[1]{\textsc{Correlate($#1$)}}
\newcommand{\Z}{\bbbz}
\newcommand{\Zpk}{\bbbz_{p^k}}
\newcommand{\Zp}[1]{\bbbz_{p^{#1}}}
\newcommand{\rep}[1]{#1^\textit{rep}}
\begin{document}
\title{On algorithms to find $p$-ordering}
\author{Aditya Gulati \and
Sayak Chakrabarti \and
Rajat Mittal}
\institute{IIT Kanpur}
\maketitle
\begin{abstract}

The concept of \emph{$p$-ordering} for a prime $p$ was introduced by Manjul Bhargava (in his PhD thesis) to develop a generalized factorial function over an arbitrary subset of integers. This notion of $p$-ordering provides a representation of polynomials modulo prime powers, and has been used to prove properties of roots sets modulo prime powers. We focus on the complexity of finding a $p$-ordering given a prime $p$, an exponent $k$ and a subset of integers modulo $p^k$.

Our first algorithm gives a $p$-ordering for set of size $n$ in time $\ot(nk\log p)$, where set is considered modulo $p^k$. The subsets modulo $p^k$ can be represented succinctly using the notion of representative roots (Panayi, PhD Thesis, 1995; Dwivedi et.al, ISSAC, 2019); a natural question would be, can we find a $p$-ordering more efficiently given this succinct representation. Our second algorithm achieves precisely that, we give a $p$-ordering in time $\ot(d^2k\log p + nk \log p + nd)$, where $d$ is the size of the succinct representation and $n$ is the required length of the $p$-ordering. Another contribution that we make is to compute the structure of roots sets for prime powers $p^k$, when $k$ is small. The number of root sets have been given in the previous work (Dearden and Metzger, Eur. J. Comb., 1997; Maulick, J. Comb. Theory, Ser. A, 2001), we explicitly describe all the root sets for $p^2$, $p^3$ and $p^4$.

\keywords{root-sets \and computational complexity \and $p$-ordering \and polynomials \and prime powers}
\end{abstract}
    \input{introc1.tex}
    \input{prereqc2.tex}
    \input{algo-setc2.tex}
    \input{algo-rep-rootsc1.tex}
    \input{struct-root-setc1.tex}
    
    \bibliographystyle{alpha}
    \bibliography{references.bib}    
\newpage
    \begin{subappendices}
\renewcommand{\thesection}{\Alph{section}}
\input{appendix-naive.tex}
\input{appendix-observations.tex}
\input{appendix-heap.tex}
\section{Correctness and Complexity of Algorithm~\ref{alg:MYALG}}
\input{appendix-sec3-proof1.tex}
\input{appendix-sec3-proof2.tex}
\input{appendix-time-complexity-algo1}
\section{Correctness and Complexity of Algorithm~\ref{alg:rep_alg}}
\input{appendix-rep-root-valuation.tex}
\input{appendix-time-complexity-rep-algo1}
\input{appendix-struct-root-setc.tex}
\end{subappendices}

\end{document}

%% file: introc1.tex
\section{Introduction}
\label{sec:introduction}

Polynomials over finite fields have played a crucial role in computer science with impact on diverse areas like error correcting codes~\cite{BC60,Hoc59,RS60,Mad97}, cryptography~\cite{CR85,Odl85,Len91}, computational number theory~\cite{AL86,AKS04} and computer algebra \cite{Zas69,LLL82}. Mathematicians have studied almost all aspects of these polynomials; factorization of polynomials, roots of a polynomial and polynomials being irreducible or not are some of the natural questions in this area. There is lot of structure over finite field; we can deterministically count roots and find if a polynomials is irreducible in polynomial time~\cite{LN97}. Not just that, we also have efficient randomized algorithms for the problem of factorizing polynomials over finite fields~\cite{CZ81,Ber70}.

The situation changes drastically if we look at rings instead of fields. Focusing our attention on the case of numbers modulo a prime power (ring $\Zpk$, for a prime $p$ and a natural number $k\geq 2$) instead of numbers modulo a prime (field $\mathbb{F}_p$), we don't even have unique factorization and the fact that the number of roots are bounded by the degree of the polynomial. Still, there has been some interesting work in last few decades. Maulik~\cite{Mau01} showed bound on number of roots sets, sets which are roots for a polynomial modulo a prime power. There has been some recent works giving a randomized algorithm for root finding~\cite{BLQ13} and a deterministic algorithm for root counting~\cite{DMS19,CGRW19}. 

The concept of \textit{$p$-ordering} and \textit{$p$-sequences} for a prime $p$, introduced by Bhargava~\cite{Bha97}, is an important tool in studying the properties of roots sets and polynomials over powers of prime $p$~\cite{Mau01,Bha97,Bha09}. Bhargava's main motivation to introduce $p$-ordering was to generalize the concept of factorials ($n!$ for $n \geq 0 \in \Z$) from the set of integers to any subset of integers. He was able to show that many number-theoretic properties of this factorial function (like the product of any $k$ consecutive non-negative integers is divisible by $k!$) remain preserved even with the generalized definition for a subset of integers~\cite{Bha00}. 

For polynomials, $p$-ordering allows us to give a convenient basis for representing polynomials on a subset of ring $\Zpk$. One of the interesting problem for polynomials over rings, of the kind $\Zpk$, is to find the allowed \textit{root sets} (Definition~\ref{def:root-set}). Maulik~\cite{Mau01} was able to use this representation of polynomials over $\Zpk$ (from $p$-ordering) to give asymptotic estimates on the number of root sets modulo a prime power $p^k$; he also gave a recursive formula for root counting.      

\paragraph{Our contributions}
While a lot of work has been done on studying the properties of $p$-orderings~\cite{Mau01,Joh09,Bha09}, there's effectively no work on finding the complexity of the problem: given a subset of numbers modulo a prime power, find a $p$-ordering. Our main contribution is to look at the computational complexity of finding $p$-ordering in different settings. We also classify and count the root-sets for $\Zpk$, when $k \leq 4$, by looking at their symmetric structure.

\begin{itemize}

\item[\textbullet] \textit{$p$-ordering for a general set:} Suppose, we want to find the $p$-ordering of a set $S\subseteq \Zpk$, such that, $|S| = n$. A naive approach, given in Bhargava~\cite{Bha97}, gives a $\ot{(n^3k\log(p))}$ time algorithm. We exploit the recursive structure of $p$-orderings and optimize the resulting algorithm using data structures. 
These optimizations allow us to give an algorithm that works in $\ot{(nk\log(p))}$ time. The details of the algorithm, its correctness and time complexity is given in Section~\ref{sec:algo_set}.

\item[\textbullet] \textit{$p$-ordering for a subset in representative root representation:} A polynomial of degree $d$ in $\Zpk$ can have exponentially many roots, but they can have at most $d$ \textit{representative roots}~\cite{Pan95,DMS19,BLQ13} giving a succinct representation. The natural question is, can we have an efficient algorithm for finding a $p$-ordering where the complexity scales according to the number of representative roots and not the size of the complete set. We answer this in affirmative, and provide an algorithm which works in $\ot{(d^2k\log{p}+nk\log{p}+nd)}$ time, where $d$ is the number of representative roots and $n$ is the length of $p$-ordering. The details of this algorithm and its analysis are presented in Section~\ref{sec:algo_rep_roots}. 

\item[\textbullet] \textit{Roots sets for small powers:} A polynomial in $\Zpk$, even with small degree, can have exponentially large number of roots.
But not all subsets of $\Zpk$ are a root-set for some polynomial.
The number of root-sets for the first few values of $k$ were calculated numerically by Dearden and Metzger~\cite{DM97}.
Building on previous work, Maulik~\cite{Mau01} produced an upper bound on the number of root-sets for any $p$ and $k$.
He also gave a recursive formula for the exact number of root-sets using the symmetries in their structure. We look at the structure of these root sets and completely classify all possible root-sets for $k\leq 4$.
In Section~\ref{sec:root-set-structure}, we discuss and distinctly describe all the root sets in $\Zp{2}$, $\Zp{3}$ and $\Zp{4}$. 
\end{itemize}

%% file: prereqc2.tex
\section{Preliminaries} \label{sec:preliminaries}
Our primary goal is to find a $p$-ordering of a given set $S \subseteq \Zpk$, for a given prime $p$ and an integer $k > 0$. Since the input size is polynomial in $|S|, \log p, k$; an efficient algorithm should run in time polynomial in these parameters. For the sake of clarity, $\log k$ factors will be ignored from complexity calculations; this omission will be expressed by using notation $\ot$ instead of $\mathcal{O}$ in time complexity. We also use $[n]$ for the set $\{0,1,\dots,n-1\}$.

We begin by defining the valuation of an integer modulo a prime $p$.
\begin{definition}\label{def:valuation}
Let $p$ be a prime and $a\neq 0$ be an integer. The \emph{valuation} of the integer $a$ modulo $p$, denoted $v_p(a)$, is the integer $v$ such that $p^v\mid a$ but $p^{v+1}\nmid a$. We also define $w_p(a) := p^{v_p(a)}$.

If $a=0$ then both, $v_p(a)$ and $w_p(a)$, are defined to be $\infty$.
\end{definition}

\begin{definition}  \label{def:set-operation}
For any ring $S$ with the usual operations $+$ and $\ast$, we have
\[S+a := \{x+a \mid x\in S\}\textit{\quad and\quad }a*S := \{a \ast x \mid x\in S\}\]
\end{definition}  

\begin{definition}
\label{def:root-set}
A given set $S$ is called a \emph{root set} in a ring $R$ if there is a polynomial $f(x) \in R[x]$, whose roots in $R$ are exactly the elements of $S$.
\end{definition}

\paragraph{Representative Roots:} 
The notion of representative roots in the ring $\Zpk$ has been used to concisely represent roots of a polynomial~\cite{Pan95,DMS19,BLQ13}.

\begin{definition} The representative root $(a + p^i\ast)$ is a subset of $\Zpk$, $$a + p^i\ast:=\{a+p^iy\mid y\in \Zp{k-i-1}\}$$ 
\end{definition} \label{def:rep-root}
Extending, a set $S=\{r_1, \cdots, r_l\}$ of representative roots correspond to $ \bigcup\limits_{i=1}^l r_i \subseteq \Zpk$. 
Conversely, we show that an $S \subseteq \Zpk$ can be uniquely represented by representative roots.

\begin{definition} \label{def:minimal_definition}
Let $S\subseteq \Zpk$, then the set of representative roots $\rep{S}=\{r_1 = \beta_1+p^{k_1}\ast, r_2 = \beta_2+p^{k_2}\ast,...,r_l = \beta_l+p^{k_l}\ast\}$  is said to be a minimal root set representation of $S$ if
\begin{enumerate}
    \item $S = \bigcup\limits_{i=1}^{l} r_i$,
    \item $\nexists~r_i, r_j \in \rep{S}: r_i \subseteq r_j$,
    \item $\forall i: \bigcup\limits_{b\in [p]} \left( r_i+p^{k_i-1}\cdot b \right)  \nsubseteq S$
\end{enumerate}
\end{definition}

\begin{theorem}\label{thm:uniqueness_of_representation}
Given any set $S\subseteq \Zpk$, the minimal root set representation of $S$ is unique.
\end{theorem}
\begin{proof}
For the sake of contradiction, let $\rep{S}$ and $\widehat{\rep{S}}$ be two different minimal representations of a set $S$. 
There exists an $a\in S$ such that it belongs to both representations, $r \in \rep{S}$ and $\widehat{r} \in \widehat{\rep{S}}$ and $r \neq \widehat{r}$. Then $r$ can be written as $a+p^{k_1}\ast$ and $\widehat{r}$ can be written as $a+p^{k_2}\ast$.

By Observation~\ref{thm:intersect_theorem}, $r\cap \widehat{r}\neq \emptyset$ implies $r \subset \widehat{r}$ or $\widehat{r} \subset r$. Without loss of generality, let $\widehat{r} \subset r$ (equivalently $k_1 < k_2$).

From $\widehat{r} \subset r$ and $k_1 < k_2$, $(\widehat{r} + b\cdot p^{k_2-1}) \subseteq r$ for all $b \in [p]$. Using $r\subseteq S$, we get $$\bigcup\limits_{b\in [p]} \left(\widehat{r} + b\cdot p^{k_2-1}\right) \subseteq S,$$ contradicting minimality of $\widehat{\rep{S}}$. 

 \QEDA
\end{proof}
\paragraph{$p$-ordering and $p$-sequence}
Bhargava~\cite{Bha97} introduced the concept of $p$-ordering
for any subset of a Dedekind domain, we restrict to the rings of the form $\Zpk$~\cite{Bha00}.

\begin{definition}[\cite{Bha97}]
\label{def:$p$-ordering}
$p$-ordering on a subset $S$ of $\Zpk$ is defined inductively.
\begin{enumerate}
    \item Choose any element $a_0 \in S$ as the first element of the sequence.
    \item Given an ordering $a_0, a_1, \dots a_{i-1}$ up to $i-1$, choose $a_i \in S \backslash \{a_0, a_1 \dots a_{i-1}\}$ which minimizes $v_p((a_i-a_0)(a_i-a_1)\dots (a_i-a_{i-1}))$.
\end{enumerate}

\end{definition}
The $i$-th element of the associated \textit{$p$-sequence} for a $p$-ordering $a_0, a_1, \dots a_n$ is defined by 
\[v_p(S,i) = \begin{cases} 
0 \qquad  i = 0, \\
w_p((a_i-a_0)(a_i-a_1) \dots (a_i-a_{i-1})) \qquad  i > 0.
\end{cases}
\]. 
In the $(i+1)$-th step, let $x \in S\setminus\{a_0,a_1,...,a_{i-1}\}$ then the value  $v_p((x-a_0)(x-a_1) \dots (x-a_{i-1}))$ is denoted as the \textit{$p$-value} of $x$ at that step. If the step is clear from context, we call the $p$-value of that element at that step as its $p$-value.

The $p$-ordering on a subset of $\Z$ can be defined similarly. Bhargava surprisingly proved the following theorem.
\begin{theorem}[\cite{Bha97}]\label{thm:p_uniqueness}
For any two $p$-orderings of a subset $S \subseteq \Z$ and a prime $p$, the associated $p$-sequences are same.
\end{theorem}

Few observations about $p$-orderings/$p$-sequences/ representative roots are listed in Appendix~\ref{apx:observations}. We also use a min-heap data structure to optimize our algorithm, details of min-heap are given in Appendix~\ref{apx:heap}.

%% file: algo-setc2.tex
\section{Algorithm to find p-ordering on a given set}
\label{sec:algo_set}
The naive algorithm for finding the $p$-ordering, from its definition, has time complexity $\ot(n^3k\log(p))$ (Appendix~\ref{apx:naive-algo}). The main result of this section is the following theorem.
\begin{theorem} \label{thm:main-thm-p-ordering}
Given a set $S \subseteq \Z$, a prime $p$ and an integer $k$ such that each element of $S$ is less than $p^k$, we can find a $p$-ordering on this set in $\ot(nk\log p)$ time. 
\end{theorem}


The proof of the theorem follows by constructing an algorithm to find the $p$-ordering.

\paragraph{Outline of the algorithm}
We use the recursive structure of $p$-ordering given by Maulik \cite{Mau01}. Crucially, to find the $p$-value of an element $a$ at each step, we only need to look at elements congruent $\bmod{p}$ to $a$ (Observation~\ref{thm:merge}). 

Suppose $S_j$ is the set of elements of $S$ congruent to $j \bmod p$. By the observation above, our algorithm constructs the $p$-ordering of set $S$ by merging the $p$-ordering of $S_j$'s. Given a $p$-ordering up to some step, the next element for the $p$-ordering of $S$ is computed by just comparing the first elements in $p$-ordering of $S_j$'s (not present in the  already computed $p$-ordering). The $p$-ordering of translated $S_j$'s is computed recursively (Observation~\ref{thm:translate_theorem}).

While merging the $p$-orderings on each of the $S_i$'s, at each step we need to extract and remove the element with the minimum $p$-value over all $S_j$'s and replace it with the next element from the $p$-ordering on the same set $S_j$. Naively, it would need to find the minimum over all $p$ number of elements taking $\ot{(p)}$ time. Instead, we use min heap data structure, using only $\ot{(\log p)}$ time for extraction and insertion of elements. 

Each node of the min-heap($H$) consists of the following values, 
\begin{enumerate}
    \item $p\_value$: contains $p$-value of the element when added to $p$-ordering,
    \item $set$: contains the index of the set $S_0,S_1,...,S_{p-1}$ element belongs to, and
    \item $value$: contains the value of the element.
\end{enumerate}

These values are used to preserve the properties of the data structures used. 
With above intuition in mind, we develop Algorithm~\ref{alg:MYALG} to find the $p$-ordering on a subset of $\Z$.
\input{algo-setc-fig.tex}
\subsection{Proof of Theorem~\ref{thm:main-thm-p-ordering}}
To prove the correctness of Algorithm~\ref{alg:MYALG}, we need two results: $\merge{ }$ procedure works and valuation is computed correctly in the algorithm. 

\begin{theorem}[Correctness of \merge{ }] \label{merge_algorithm}
In Algorithm~\ref{alg:MYALG}, given $S$ be a subset of integers, let for $k \in \{0,1,...,p-1\}$, $S_k = \{s\in S \mid s \equiv k\pmod{p}\}$, then given a $p$-ordering on each of the $S_k$'s,  \merge{S_0,S_1,...,S_{p-1}} gives a valid $p$-ordering on $S$. 
\end{theorem} 
\begin{proof}[Proof outline]
We start with $p$-orderings on each of the non-empty sets $(S_0,S_1,...,S_{p-1})$, and create a heap taking the first element of each of these $p$-ordering. 
At each successive step, we pick the element in the heap with minimum $p$-value to add to the $p$-ordering, and insert the next element from the corresponding $S_j$ to the heap.

We know that the valuation of any element in the combined $p$-ordering is going to be equal to their valuation in the $p$-ordering over the set $S_j$ containing them (by Observation~\ref{thm:merge}). If we show that at each step the element chosen has the least valuation out of all the elements left \merge{ } works correctly. We prove this by getting a contradiction if any element other than the ones obtained from the min heap is selected by showing the p\_value will be greater than what we get from \merge{ }.\QEDA

The details of the proof can be found in Appendix \ref{apx:merge_algorithm}.
\end{proof}

\begin{theorem}[Correctness of valuations] \label{p_value_preserve_theorem}
In Algorithm~\ref{alg:MYALG}, let $S$ be a subset of integers, then \findpordering{S} gives a valid $p$-values for all elements of $S$.
\end{theorem}
\begin{proof}[Proof outline]
The proof requires two parts: \merge{ } preserves valuation and changes in the valuation due to \emph{translation} does not induce errors. 
\begin{itemize}
    \item[\textbullet] To prove that \merge{ } preserves valuation, we make use of the fact that the combined $p$-ordering after merge has the individual $p$-orderings as a sub-sequence. Hence, the valuation of each element in the combined $p$-ordering is going to be equal to the valuation in the individual $p$-ordering (by Observation~\ref{thm:merge}). Hence, \merge{ } preserves valuations.
    \item[\textbullet] We show that the change in valuations due to translation (Step 24) are corrected (Step 27). This is easy to show by just updating the valuation according to Observation~\ref{thm:p_value_translate_theorem}.
\end{itemize}  
Hence, valuations are correct maintained throughout the algorithm. \QEDA

The details of the proof can be found in Appendix~\ref{apx:p_value_preserve_theorem}.
\end{proof}
Using the above two theorems, we prove the correctness of Algorithm~\ref{alg:MYALG}.

\begin{proof}[Proof of Theorem~\ref{thm:main-thm-p-ordering}] 
For the base case, if $S$ is a singleton, then the $p$-ordering over it is just a single element which is also what \findpordering{S} gives. 
Let \findpordering{ } works for $|S|<k$, if we show it works for $|S|=k$, then by induction, \findpordering{ } works for sets of arbitrary sizes.

Let $|S|=k$, then when we break the set into $S_0, S_1,...,S_{p-1}$ (Steps 20-22), either all element belong in a single $S_i$ or get distributed into multiple sets. We can argue that if all the elements fall into the same group, then when we keep calling recursion (Step 24), after some point set breaks into multiple $S_i$'s. Since, by Observation~\ref{thm:translate_theorem}, we know that the $p$-ordering on reduced elements is preserved, we'll get the correct $p$-ordering on the original set. Hence, we only need to prove this for the later case.

Since all the element of the set $S_i$ follow $\forall y \in S_i$, $y\equiv i\bmod{p}$, hence $\forall y \in S_i$, $p \mid (y-i)$, this implies $(S_i-i)/p \subset \Z$. Hence, \findpordering{(S_i-i)/p} gives a $p$-ordering on $(S_i-i)/p$ with the correct valuations associated with each element (Theorem~\ref{p_value_preserve_theorem}).

From Observation~\ref{thm:translate_theorem}, we know that if $(a_0,a_1,...)$ be a $p$-ordering on some set $A$, then $(p*a_0+i,p*a_1+i,...)$ be a $p$-ordering on $p*A+i$. Since, \findpordering{(S_i-i)/p} is a $p$-ordering on $(S_i-i)/p$, then $p*\findpordering{(S_i-i)/p}+x$ is a $p$-ordering on $S_i$ (Step 26).

Next, since we have valid $p$-orderings on $S_0,S_1,...,S_{p-1}$, \merge{S_0,S_1,...,S_{p-1}} returns a valid $p$-ordering on $S$ (Theorem \ref{merge_algorithm}).

By induction, our algorithm returns a valid $p$-ordering on any subset of integers. 

If each element of $S$ is less than $p^k$, then $p$-ordering on set $S$ requires $\ot(nk\log p)$ time (Theorem~\ref{thm:time-algo-1} of Appendix~\ref{apx:algo1-time}).
\QEDA
\end{proof}


%% file: algo-setc-fig.tex
\begin{algorithm}[ht]
\caption{Find p-ordering}
\label{alg:MYALG}
\begin{algorithmic}[1]
\Procedure{\merge{S_0,S_1,...,S_{p-1}}}{}
\State $S \gets [\ ]$
\For{$i \in [0,1,...,p-1]$}
\For{$j \in [0,...,len(S_{i})-1]$}
\State $S_{i}[j].set \gets i$
\EndFor
\EndFor
\State $i_0,i_1,i_2,...i_{p-1} \gets (0,0,...,0)$
\State $H \gets \createminheap{node=\{S_0[i_0],S_1[i_1],...,S_{p-1}[i_{p-1}]\}, key=p\_value}$
\While{H.IsEmpty()!=true}
\State $a \gets \extractmin{H}$
\State $j \gets a.set$
\If{$i_j < len(S_j)$}
\State $i_j \gets i_j + 1$
\State $\insertheap{H}{S_j[i_j]}$
\EndIf
\State $S \gets a$
\EndWhile
\State \textbf{return} $S$
\EndProcedure
\Procedure{\findpordering{S}}{}
\If{length(S)==1}
\State $S[0].p\_value \gets 1$
\State \textbf{Return} $S$
\EndIf
\State $S_0,S_1,...,S_{p-1} \gets ([\ ],[\ ],...,[\ ])$
\For{$i \in S$}
\State $S_{i.value\bmod{p}}.append(i)$
\EndFor
\For{$i \in [0,1,...,p-1]$}
\State $S_i \gets \findpordering{(S_i-i)/p}$
\For{$j \in [0,...,len(S_i)-1]$}
\State $S_i[j].value \gets p*S_i[j].value+i$
\State $S_i[j].p\_value \gets S_i[j].p\_value+j$
\EndFor
\EndFor
\State $S \gets \merge{S_0,S_1,...,S_{p-1}}$
\State \textbf{return} $S$
\EndProcedure
\end{algorithmic}
\algcomment{
In Algorithm~\ref{alg:MYALG}, we use a sorted list $\mathcal{I}$ of non-empty $S_i$'s, and only iterate over $\mathcal{I}$ in steps 3-5, 23-28. Hence, decreasing the time complexities of these loops. We can create/update the list  $\mathcal{I}$ in the loop at steps 21-22.}
\end{algorithm}

%% file: algo-rep-rootsc1.tex
\section{Algorithm to find $p$-ordering on a set of representative roots}
\label{sec:algo_rep_roots}

The notion of representative roots (Definition~\ref{def:rep-root}) allows us to represent an exponentially large subset of $\Zpk$ succinctly. Further imposing a few simple conditions on this representation, namely the minimal representation (Definition~\ref{def:minimal_definition}), our subset is represented in a unique way (Theorem~\ref{thm:uniqueness_of_representation}). A natural question arises, can we efficiently find a $p$-ordering given a set in terms of representative roots? 

In this section we show that the answer is affirmative by constructing an efficient algorithm in terms of the size of the succinct representation.
\begin{theorem}
\label{thm:main-thm-rep-root}
Given a set $S \subset \Zpk$, for a prime $p$ and an integer $k$, that can be represented in terms of $d$ representative roots, we can efficiently find a $p$-ordering of length $n$ for $S$ in $\ot(d^2k\log p + nk \log p + np)$ time.
\end{theorem}

The proof of the theorem follows by constructing an algorithm to find the $p$-ordering given a set in representative root notation. We can assume that the representative roots are disjoint. If they are not, one representative root will be contained in another (Observation~\ref{thm:intersect_theorem}), all such occurrences can be deleted in $\ot(d^2k\log p)$ time.

\paragraph{Outline of the algorithm}
The important observation is, we already have a natural $p$-ordering defined on a representative root (Observation~\ref{thm:merge_rep_root}).
Since a $p$-ordering on each representative root is already known, we just need to find a way to merge them. Merging was easy in Algorithm~\ref{alg:MYALG} because progress in any one of the $p$-ordering of an $S_j$ did not effect the $p$-value of an element outside $S_j$. However, in this case the exact increase in the $p$-value is known by Observation~\ref{thm:rep_root_interaction}.

Let $d$ be the number of representative roots, we maintain an array of size $d$ to keep the valuations that we would get whenever we add the next element from a representative root. To update the $i$-th value of this array when an element from the $j$-th representative root is added, we simply add the value $v_p(\beta_i-\beta_j)$($i\neq j$). Hence, at each step we find the minimum value in this array (in $\ot($d$)$) and add it to the combined $p$-ordering (in $\ot($1$)$) and we update all the $d$ values in this array (in $\ot($d$)$). We repeat this process till we get the $p$-ordering of the desired length.

With the above intuition in mind, we develop Algorithm~\ref{alg:rep_alg} to find the $p$-ordering of length $n$ on a subset $S$ of $\Zpk$ given in representative root representation. \\

\input{algo-rep-rootsc-fig.tex}
\subsection{Proof of Correctness}
To prove the correctness of our algorithm, we first prove that valuations are correctly maintained.
\begin{theorem}
\label{thm:correct_valuation_array}
In Algorithm~\ref{alg:rep_alg}, \findpordering{S,n} maintains the correct valuations on the set $S$ of representative roots in $valuations$ at every iteration of the loop.
\end{theorem}
\begin{proof}[Proof outline]
All elements have 0 valuation at the beginning (Step 17). Also, adding an element from the $i$-th representative root increases the valuation of the $j$-th representative root by $corr(i,j)$ (Step 33) for $i\neq j$ (Observation~\ref{thm:rep_root_interaction}). The increase for the next element of $i$ is exponent times the increase in $p$-sequence of $\Zpk$ (Step 30) (Observation~\ref{thm:merge_rep_root}). So, we correctly update the valuations array in each iteration. \QEDA

A detailed proof can be found in Appendix \ref{apx:valuation-rep-root}.
\end{proof}
\begin{proof}[Proof of Theorem~\ref{thm:main-thm-rep-root}]
By the definition of $p$-ordering we know that at each iteration if we choose the element with the least valuation then we get a valid $p$-ordering. By Theorem~\ref{thm:correct_valuation_array}, we know that $valuations$ array has the correct next valuations. Hence, to find the representative root with gives the least valuation, we find the index of the minimum element in $valuations$.

To add the next value to the $p$-ordering, we use Observation~\ref{thm:merge_rep_root} to find the next element in the $p$-ordering on the representative root. Hence, the element added has the least valuation. Hence, \findpordering{S,n} returns the correct $p$-ordering.

 If $S$ contains $d$ representative roots of $\Zpk$, then Algorithm~\ref{alg:rep_alg} finds $p$-ordering on $S$ up to length $n$ in $\ot{(d^2k\log{p}+nk\log{p}+nd)}$ time (Theorem~\ref{thm:time-rep-algo} of Appendix~\ref{apx:rep-algo-time}). \QEDA
\end{proof}

%% file: algo-rep-rootsc-fig.tex
\begin{algorithm}[!ht]
\caption{Find p-ordering from minimal notation}
\label{alg:rep_alg}
\begin{algorithmic}[1]
\Procedure{\correlate{S}}{}
\State $Corr \gets [0]_{len(S)\times len(S)}$
\State $Corr \gets [0]_{len(S)\times len(S)}$
\For{$j \in [1,...,len(S)]$}
\For{$k \in [1,...,len(S)]$}
\State $Corr[j][k] \gets v_p(S[j].value-S[k].value)$
\EndFor
\EndFor
\State \textbf{return} $Corr$
\EndProcedure
\Procedure{\pexpincrease{n}}{}
\State $v_p(1) \gets 1$
\For{$j \in [1,...,n]$}
\State $v_p((j+1)!) \gets v_p(j+1)\ast v_p(j!)$
\State $p\_exponent[j] \gets v_p((j+1)!)-v_p(j!)$
\EndFor
\State \textbf{return} $p\_exponent$
\EndProcedure
\Procedure{\findpordering{S, n}}{}
\State $corr \gets \correlate{S}$
\State $increase \gets \pexpincrease{n}$
\State $valuations \gets [0]_{|S|}$
\State $p\_ordering \gets \{\}$
\State $i_1, i_2 \dots i_{|S|} \gets 0$
\For{$i \in \{1, 2, \dots n\}$}
    \State $min \gets \infty$
    \State $min\_index \gets \infty$
    \For{$j \in [1,...,len(S)]$}
        \If{$valuations[j]<min$}
            \State $min \gets valuations[j]$
            \State $min\_index \gets j$
        \EndIf
    \EndFor
    \State $p\_ordering.append(S[min\_index].value+p^{S[min\_index].exponent}\ast i_{min\_index})$
    \For{$j \in [1,...,len(S)]$}
        \If{$j=min\_index$}
            \State $valuations[j] \gets valuations[j] + S[min\_index].exponent\ast increase[i_j]$
        \Else
            \State $valuations[j] \gets valuations[j] + corr(min\_index,j)$
        \EndIf
    \EndFor
    \State $i_{min\_index} \gets i_{min\_index} + 1$
\EndFor
\State \textbf{return} $p\_ordering$

\EndProcedure
\end{algorithmic}
\end{algorithm}

%% file: struct-root-setc1.tex
\section{Structure of root sets for a given $k$}
\label{sec:root-set-structure}
We know that $\Zpk$ is not a unique factorization domain. In fact, even small degree polynomials can have exponentially large number of roots.
Interestingly, not all subsets of $\Zpk$ can be a root set (Definition~\ref{def:root-set}). 
Dearden and Metzger~\cite{DM97} showed that $R$ is a root-set iff $R_j=\{r\in R \mid r \equiv j\pmod{p}\}$ is also a root-set for all $j\in [p]$. The size and structure of $R_j$ is symmetric for all $j$. Let $N_{p^k}$ be the number of possible $R_j$'s, then total number of possible root-sets become $(N_{p^k})^p$~\cite{DM97}. In this section, we discuss and describe all possible $R_j$'s in $Z_{p^2}$, $Z_{p^3}$ and $Z_{p^4}$.


Let $S_j = \{s\in \Zpk \mid s \equiv j\pmod{p}\}$, we take the following approach to find all possible root-sets $R_j$'s. Given an $R_j$, define $R=\{(r-j)/p: r\in R_j\}$ to be the translated copy. We show that if $R$ contains at least $k$ many distinct residue classes $\bmod p$, then $R_j = S_j$ (Observation~\ref{thm:minimal-complete-tree}). We exhaustively cover all the other cases, when $R$ contains less than $k$ residue classes (possible because $k$ is small). 

\subsection{$k=2$}\label{sec:struct-2k}
We find that the root set $R_j$ can only take the following structures (details in Appendix~\ref{apx:struct-2k}).
\begin{enumerate}
    \item $\bm{1}$ root-set is the complete sub-tree under $j$ (more than 1 residue class), equivalently \\
    \vspace*{-\baselineskip}$$R_j=j+p\cdot\ast.$$
    \item $\bm{p}$ root-sets are a single element congruent to $j \bmod{p}$ (1 residue class), equivalently \\
    \vspace*{-\baselineskip}$$R_j=j+p\cdot\alpha\text{, for } \alpha\in [p].$$
    \item $\bm{1}$ root-set is empty (no residue classes), equivalently \\
    \vspace*{-\baselineskip} $$R_j=\emptyset.$$ 
\end{enumerate}
Hence, total root-sets,  $N_{p^2}=p+2$.
\subsection{$k=3$}\label{sec:struct-3k}

Similar to $k=2$, the root sets $R_j$ can only take the following structure (details in Appendix~\ref{apx:struct-3k}).
\begin{enumerate}
    \item $\bm{1}$ root-set is the complete sub-tree, equivalently \\
    \vspace*{-\baselineskip}$$R_j=j+p\cdot\ast.$$
    \item $\bm{\frac{p(p-1)}{2}}$ root-sets are the union of 2 sub-trees different at the level $p^1$, equivalently \\
    \vspace*{-\baselineskip} $$R_j=(j+p\cdot\alpha_1+p^2\ast)\cup(j+p\cdot\alpha_2+p^2\ast)\text{, for } \alpha_1\neq\alpha_2\in [p].$$ 
    \item $\bm{p}$ root-sets are a sub-tree at the level $p^1$, equivalently \\
    \vspace*{-\baselineskip}$$R_j=j+p\cdot\alpha+p^2\ast\text{, for } \alpha\in [p].$$ 
    \item $\bm{p^2}$ root-sets are a single element congruent to $j \bmod{p}$, equivalently \\
    \vspace*{-\baselineskip}$$R_j=j+p\cdot\alpha_1+p^2\cdot\alpha_2\text{, for } \alpha_1,\alpha_2\in [p].$$ 
    \item $\bm{1}$ root-set is empty, equivalently \\
    \vspace*{-\baselineskip}$$R_j=\emptyset.$$ 
\end{enumerate}
Hence, total root-sets,  $N_{p^3}=\frac{3p^2+p+4}{2}$.

\subsection{$k=4$}\label{sec:struct-4k}

Similar to $k=2,3$, the root sets $R_j$ can only take the following structure (details in Appendix~\ref{apx:struct-4k}).
\begin{enumerate}
    \item $\bm{1}$ root-set is the complete sub-tree under $j$, equivalently \\
    \vspace*{-\baselineskip}$$R_j=j+p\cdot\ast.$$ 
    \item $\bm{\frac{p(p-1)(p-2)}{6}}$ root-sets under $j$ are the union of 3 sub-trees different at the level $p^1$, equivalently \\
    \vspace*{-\baselineskip}$$R_j=(j+p\cdot\alpha_1+p^2\ast)\cup(j+p\cdot\alpha_2+p^2\ast)\cup(j+p\cdot\alpha_3+p^2\ast)\text{, for } \alpha_1\neq\alpha_2\neq\alpha_3\in [p].$$
    \item $\bm{\frac{p(p-1)}{2}}$ root-sets are the union of 2 sub-trees different at the level $p^1$, equivalently \\
    \vspace*{-\baselineskip}$$R_j=(j+p\cdot\alpha_1+p^2\ast)\cup(j+p\cdot\alpha_2+p^2\ast)\text{, for } \alpha_1\neq\alpha_2\in [p].$$ 
    \item $\bm{p}$ root-sets are a sub-tree at the level $p^1$, equivalently \\
    \vspace*{-\baselineskip}$$R_j=j+p\cdot\alpha+p^2\ast\text{, for } \alpha\in [p].$$ 
    \item $\bm{\frac{p^3(p-1)}{2}}$ root-sets are a union of 2 sub-trees at the level $p^2$ that are different at the level $p^1$, equivalently \\
    \vspace*{-\baselineskip}$$R_j=(j+p\cdot\alpha_1+p^2\cdot\beta_1+p^3\ast)\cup(j+p\cdot\alpha_2+p^2\cdot\beta_2+p^3\ast)\text{, for } \alpha_1\neq\alpha_2\,\beta_1,\beta_2\in [p].$$ 
    \item $\bm{p^2}$ root-sets are a sub-tree at the level $p^2$, equivalently\\
    \vspace*{-\baselineskip}$$R_j=j+p\cdot\alpha_1+p^2\cdot\alpha_2+p^3\cdot\ast\text{, for } \alpha_1,\alpha_2\in [p].$$ 
    \item $\bm{p^3}$ root-sets are a single element congruent to $j \bmod{p}$, equivalently\\
    \vspace*{-\baselineskip}$$R_j=j+p\cdot\alpha_1+p^2\cdot\alpha_2+p^3\cdot\alpha_3\text{, for } \alpha_1,\alpha_2,\alpha_3\in [p].$$ 
    \item $\bm{1}$ root-set is empty, equivalently\\
    \vspace*{-\baselineskip}$$R_j=\emptyset.$$ 
\end{enumerate}
Hence, total root-sets,  $N_{p^4}=\frac{3p^4+4p^3+6p^2+5p+12}{6}$.

%% file: appendix-naive.tex
\section{Naive Algorithm to find $p$-ordering} \label{apx:naive-algo}
Given a set of integers $S \subseteq \Zpk$ we can find a $p$-ordering by naively checking the element which will give us the minimum valuation with respect to $p$ for the given expression as in Definition~\ref{def:$p$-ordering}. After we have already chosen $\{a_0, a_1, \dots a_{t-1}\}$ we choose the next element from $S \backslash \{a_0, a_1, \dots a_{t-1}\} $ such that $v_p((x-a_0)(x-a_1)\dots (x-a_{t-1}))$ is minimum. The naive approach given in \cite{Bha97} iterates over all $x$ in $S \backslash \{a_0, a_1, \dots a_{t-1}\}$ and adds the element to the $p$-ordering which gives the minimum valuation.

\paragraph{Time Complexity:}
Every time we keep on adding another element to the already existing $p$-ordering, say of length $t$. For any given value of $x \in S \backslash \{a_0, a_1, \dots a_{t-1}\}$, calculating $x-a_i$ and multiplying for every $0 \geq i < t$ takes $\mathcal{O}((n-t)t)$ operations in $\mathbb{Z}$ and since each of them are less than $p^k$ this takes $\mathcal{O}((n-t)t k\log p) \leq \mathbb{O}(n^2k\log p)$. So repeating this $n$ times gives us the time complexity $\mathcal{O}(n^3k\log p)$.

%% file: appendix-observations.tex
\section{Observations about representative roots and $p$-sequences/$p$-orderings}
\label{apx:observations}
\subsection{Representative roots}
\begin{observation} \label{thm:intersect_theorem}
Given any two representative roots $A_1=\beta_1+p^{k_1}\ast$ and $A_2=\beta_2+p^{k_2}\ast$, then either $A_1 \subseteq A_2$ or $A_2 \subseteq A_1$ or $A_1 \cap A_2 = \emptyset$.
\end{observation}
\begin{proof}


Let $A_1=\beta_1+p^{k_1}\ast$ and $A_2=\beta_2+p^{k_2}\ast$ be two root sets such that $A_1\cap A_2 = \tilde{A}\neq \emptyset$, then we show that $\tilde{A} = A_1$ or $\tilde{A} = A_2$. 
\paragraph{Case 1:} Let $k_1=k_2=x$. Let there is some element $a \in \tilde{A}$. Then $a \in A_1$ and $a \in A_2$, hence, $A_1$ can be defined as $A_1=a+p^x\ast$ and similarly $A_2=a+p^x\ast$. Hence, $\tilde{A} = A_1 = A_2$.
\paragraph{Case 2:} Let $k_1\neq k_2$, then without loss of generality, let's assume that $k_1<k_2$. Let there is some element $a \in \tilde{A}$. Then, $A_1$ can be defined as $A_1=a+p^{k_1}\ast$ and similarly $A_2=a+p^{k_2}\ast$.

Let $b \in A_2$, then $b = a + p^{k_2}y$ for some $y$. Now, we know that the elements of $A_1$ are of the form $a+p^{k_1}\ast$. Hence, putting $\ast = p^{k_2-k_1}y$, we get $a+p^{k_1}\cdot (p^{k_2-k_1}y)=b$, hence $b\in A_1$. Hence, $A_2\subset A_1$. Hence, $\tilde{A}=A_2$. \QEDA
\end{proof}

\begin{observation}
\label{thm:rep_root_interaction} Let $a_1\in\beta_1+p^{k_1}\ast$ and $a_2\in\beta_2+p^{k_2}\ast$ be any 2 elements of the representative roots $\beta_1+p^{k_1}\ast$ and $\beta_2+p^{k_2}\ast$ respectively, for $\beta \neq \alpha_2$, then, $$w_p(a_1-a_2)=w_p(\beta_1-\beta_2).$$
\end{observation}
\begin{proof}
We have $2$ representative roots of the form $\beta_1 + p^{k_1}*$ and $\beta_2 + p^{k_2}*$. WLOG let us assume that $k_1 \leq k_2$ and $\beta_1 \in \Zp{k_1}, \beta_2 \in \Zp{k_2}$. \\
We definitely have that these two are different representative roots. So if the first $k_1$ elements of the $p$-adic expansion of $\beta_2$ are equal then the second representative root will be contained in the first, as the $*$ portion of the first contains all the values of the second representative root as well as its subset. So for them to be different representative roots, $p^{k_1} \nmid \beta_2-\beta_1$. 
Let $v_p(\beta_1-\beta_2) = t$, $t < k_1$, then for any value $y_1, y_2$ in the respective $*$ sets, we will have $p^t | (\beta_1 +p^{k_1}y_1)-(\beta_2 + p^{k_2}y_2)$. Note that since $k_1 > v_p(\beta_1-\beta_2)$ we have $v_p((\beta_1 +p^{k_1}y_1)-(\beta_2 + p^{k_2}y_2)) \geq v_p(\beta_1-\beta_2)$.\\ Now, since $p^t | \beta_1-\beta_2$ and $p^{t+1} \nmid \beta_1 - \beta_2$ we can write $\beta_1-\beta_2 = p^t(a+pb)$ for $a, b \in \Zpk$ where $a \in \{1, 2, \dots p-1\}$. This implies that  if $p^{t+1} | (\beta_1 +p^{k_1}y_1)-(\beta_2 + p^{k_2}y_2)$, it means $p^{t+1} | p^t(a+pb) + p^{k_1}(y_1 - p^{k_2-k_2}y_2) \implies p | a + p (\cdots)$ which can not be true as $p \nmid a$. So for any value of $y_1, y_2$ in the respective $*$ sets of their corresponding representative roots, we will have $v_p((\beta_1 +p^{k_1}y_1)-(\beta_2 + p^{k_2}y_2)) = v_p(\beta_1 - \beta_2)$.\\
Conversely if $\exists y_1, y_2$ such that $v_p((\beta_1 +p^{k_1}y_1)-(\beta_2 + p^{k_2}y_2)) > v_p(\beta_1-\beta_2)$, let $l = v_p((\beta_1 +p^{k_1}y_1)-(\beta_2 + p^{k_2}y_2)) \geq k_1$. Then we have $v_p(\beta_1-\beta_2) \leq l-1$. So if $p^l | (\beta_1 +p^{k_1}y_1)-(\beta_2 + p^{k_2}y_2) \implies p^l | \beta_1-\beta_2$ (as $l \leq k_1$). This is a contradiction as $l > v_p(\beta_1-\beta_2)$.\\
This completes the proof of Observation~\ref{thm:rep_root_interaction}. \QEDA
\end{proof}

\subsection{$p$-ordering and $p$-sequence}
\begin{observation} [\cite{Mau01}] \label{thm:merge}
Let $S$ be a subset of integers, let $S_j=\{s\in S \mid s \equiv j\pmod{p}\}$ for $j=0,1,...,p-1$, then for any $x\in \Z$, s.t. $x\equiv j\pmod{p}$, 
\begin{equation}
    w_p\left(\displaystyle \prod_{a_i \in S} (x-a_i)\right) = w_p\left(\displaystyle \prod_{a_i \in S_j} (x-a_i)\right).
\end{equation}
\end{observation}
\begin{observation} \label{thm:translate_theorem}
Let $S$ be a subset of integers, let ($a_0,a_1,a_2,...$) be a $p$-ordering on $S$, then
\begin{enumerate}
    \item For any $x \in \Z$, ($a_0+x,a_1+x,a_2+x,...$) is a $p$-ordering on $S+x$.
    \item For any $x \in \Z$, ($x*a_0,x*a_1,x*a_2,...$) is a $p$-ordering on $x*S$.
\end{enumerate}
\end{observation}

\begin{observation} \label{thm:p_value_translate_theorem}
Let $S$ be a subset of integers, let $(a_0,a_1,a_2,...)$ be a $p$-ordering on $S$. Then, for any $x \in \Z$
\begin{enumerate}
    \item $v_p(x*S,k)=v_p(S,k)+k\cdot w_p(x)$.
    \item $v_p(S+x,k)=v_p(S,k)$.
\end{enumerate}
\end{observation}

\begin{observation}
\label{thm:merge_rep_root}
Let $(a_0,a_1,...)$ be a $p$-ordering on $\Zpk$, then $(\beta+a_0*p^j,\beta+a_1*p^j,\beta+a_2*p^j,...)$ is a $p$-ordering on $\beta+p^j\ast$.
\end{observation}
\begin{proof}
A simple proof of this theorem follows from Observation~\ref{thm:translate_theorem} and the fact that $1, 2, 3, \dots$ form an obvious $p$-ordering in $\Zpk$. \QEDA
\end{proof}

%% file: appendix-heap.tex
\section{Min-heap data structure}
\label{apx:heap}
A min-heap is a data structure in which each node has at most two children and exactly one parent node (except root, no parents). The defining property is that the key value of any node is equal or lesser than the key value of its children.

\begin{tikzpicture}
    \node[shape=circle,draw=black] (A) at (0,0) {3};
    \node[shape=circle,draw=black] (B) at (-1,-1) {5};
    \node[shape=circle,draw=black] (C) at (1, -1) {7};
    \node[shape=circle,draw=black] (D) at (-1.5, -2) {11};
    \node[shape=circle,draw=black] (E) at (-0.5, -2) {6};
    \node[shape=circle,draw=black] (F) at (1.5, -2) {8} ;
    \node[shape=circle, draw=black] (G) at (1, -3) {9};
    \node[shape=circle, draw=black] (H) at (2, -3) {10};

    \draw [-] (A) -- (B);
    \draw [-] (A) -- (C);
    \draw [-] (B) -- (D);
    \draw [-] (B) -- (E);
    \draw [-] (C) -- (F);
    \draw [-] (F) -- (G);
    \draw [-] (F) -- (H);
\end{tikzpicture}

We will use three standard functions on a min-heap with $n$ nodes~\cite{CLRS01}.
\begin{enumerate}
    \item \createminheap{S}: Takes a set $S$ as input and returns a min-heap with elements of $S$ as the nodes in $\ot{(n)}$.
    \item \extractmin{H}: Removes the element with the minimum key from the heap and rebalances the heap structure in $\ot{(\log(n))}$.
    \item \insertheap{H}{a}: Inserts the element $a$ into the heap $H$ in $\ot{(\log(n))}$.
\end{enumerate}

%% file: appendix-sec3-proof1.tex
\subsection{Proof of Theorem  \ref{merge_algorithm}}\label{apx:merge_algorithm}
Let $(a_0^k,a_1^k,...)$ be a $p$-ordering on each of the $S_k$'s.
We know that \merge{ } on $(S_0,S_1,...,S_{p-1})$ proceeds in such a way that elements are added to the heap in such a way that the element $a_l^k$ is added to the heap only after $\forall i<l, a_i^k$ have already been added to the $p$-ordering. Also, we know at any point, only one element from any $S_k$ can belong to the heap. 

We know that at any point, let $(a_0,a_1,...a_{k-1})$ be a $p$-ordering on $S$, then the next element $a_k$ in the $p$-ordering on $S$ if and only if $$w_p\left( \prod_{i \in \{0,1,...,k-1\}} (a_k-a_i)\right) = \min_{x \in S\setminus \{a_0,a_1,...a_{k-1}\}}\left(w_p\left( \prod_{i \in \{0,1,...,k-1\}} (x-a_i)\right)\right).$$
We know that if at each point, our pick for the next element in the $p$-ordering satisfies this condition, the $p$-ordering we get is valid.

Lets say $(a_0,a_1,...,a_i)$ be the $p$-ordering we have till now. Let elements currently the elements $(a_0^k,a_1^k,...,a_{i_k-1}^k)$ of the given $p$-ordering on set $S_k$ are currently a part of the $p$-ordering. Let the elements $(a_{i_0}^0,a_{i_1}^1,...,a_{i_{p-1}}^{p-1})$ are a part of the min-heap. 

Let when we extract the min from the min-heap, we get some value $a_{i_k}^k$. If we show that $$w_p\left( \prod_{j \in \{0,1,...,i\}} (a_{i_k}^k-a_j)\right) = \min_{x \in S\setminus\{a_0,a_1,...a_{i}\}}\left(w_p\left( \prod_{j \in \{0,1,...,i\}} (x-a_j)\right)\right),$$
then we know that $a_{i_k}^k$ is a valid next element in the $p$-ordering and hence \merge{ } gives a correct $p$-ordering on $S$.

We prove this by contradiction. Let there is an element $x \in S\setminus\{a_0,a_1,..,a_{i},a_{i_k}^k\}$ such that $$w_p\left( \prod_{j \in \{0,1,...,i\}} (x-a_j)\right) < w_p\left( \prod_{j \in \{0,1,...,i\}} (a_{i_k}^k-a_j)\right).$$ Then, we have 2 cases, either $x \in S_k$, i.e. $x \equiv k \bmod{p} \equiv a_{i_k}^k \bmod{p}$ or $x \in S_l$ for some $l \neq k$, i.e. $x \equiv l \bmod{p} \not\equiv a_{i_k}^k \bmod{p}$.

\begin{description}

\item[Case 1: $x \in S_k$].
   
Our assumption is that $$w_p\left( \prod_{j \in \{0,1,...,i\}} (x-a_j)\right) < w_p\left( \prod_{j \in \{0,1,...,i\}} (a_{i_k}^k-a_j)\right).$$ We know that from our assumption that $S_k \cap (a_0,a_1,...,a_i) = (a_0^k,a_1^k,...,a_{i_k-1}^k)$. From Observation~\ref{thm:merge}, $$w_p\left( \prod_{j \in \{0,1,...,i\}} (a_{i_k}^k-a_j)\right) = w_p\left( \prod_{j \in \{0,1,...,i_k-1\}} (a_{i_k}^k-a_j^k)\right),$$ and $$w_p\left( \prod_{j \in \{0,1,...,i\}} (x-a_j)\right) = w_p\left( \prod_{j \in \{0,1,...,i_k-1\}} (x-a_j^k)\right).$$ Since, $(a_0^k,a_1^k,...,a_{i_k-1}^k,a_{i_k}^k,...)$ is a valid $p$-ordering on $S_k$, $$w_p\left( \prod_{j \in \{0,1,...,i\}} (a_{i_k}^k-a_j)\right) \leq w_p\left( \prod_{j \in \{0,1,...,i\}} (x-a_j)\right).$$ 
But this is a contradiction.

\item[Case 2: $x \in S_l$ for some $l \neq k$].

Our assumption is that $$w_p\left( \prod_{j \in \{0,1,...,i\}} (x-a_j)\right) < w_p\left( \prod_{j \in \{0,1,...,i\}} (a_{i_k}^k-a_j)\right).$$ Let the element belonging to the set $S_l$ in the heap is $a_{i_l}^l$. We know that from our assumption that $S_l \cap (a_0,a_1,...,a_i) = (a_0^l,a_1^l,...,a_{i_l-1}^l)$. From Observation~\ref{thm:merge}, $$w_p\left( \prod_{j \in \{0,1,...,i\}} (a_{i_l}^l-a_j)\right) = w_p\left( \prod_{j \in \{0,1,...,i_l-1\}} (a_{i_l}^l-a_j^l)\right),$$ and $$w_p\left( \prod_{j \in \{0,1,...,i\}} (x-a_j)\right) = w_p\left( \prod_{j \in \{0,1,...,i_l-1\}} (x-a_j^l)\right).$$ Since, $(a_0^l,a_1^l,...,a_{i_l-1}^l,a_{i_l}^l,...)$ is a valid $p$-ordering on $S_l$, $$w_p\left( \prod_{j \in \{0,1,...,i\}} (a_{i_l}^l-a_j)\right) \leq w_p\left( \prod_{j \in \{0,1,...,i\}} (x-a_j)\right).$$ Also, since both $a_{i_l}^l$ and $a_{i_k}^k$ were part of the heap but $ExtractMin()$ procedure returned $a_{i_k}^k$, $$w_p\left( \prod_{j \in \{0,1,...,i\}} (a_{i_k}^k-a_j)\right) \leq w_p\left( \prod_{j \in \{0,1,...,i\}} (a_{i_l}^l-a_j)\right).$$ Using the above two inequalities, $$w_p\left( \prod_{j \in \{0,1,...,i\}} (a_{i_k}^k-a_j)\right) \leq w_p\left( \prod_{j \in \{0,1,...,i\}} (x-a_j)\right).$$ 
But this is a contradiction.
\end{description}

Since, we arrive at a contradiction in both the cases, hence, our assumption must be wrong. Hence, $$w_p\left( \prod_{i \in \{0,1,...,k-1\}} (a_k-a_i)\right) = \min_{x \in S\setminus \{a_0,a_1,...a_{k-1}\}}\left(w_p\left( \prod_{i \in \{0,1,...,k-1\}} (x-a_i)\right)\right).$$ Hence, our procedure \merge{ } gives a valid $p$-ordering. \QEDA

%% file: appendix-sec3-proof2.tex
\subsection{Proof of Theorem  \ref{p_value_preserve_theorem}}
\label{apx:p_value_preserve_theorem}
We prove this by induction on the size of $S$.

If $S$ is a singleton, then the $p$-ordering on $S$ is just that element. And hence the corresponding $p$-value is just $p^0=1$. \findpordering{S} sets this value to $1$ in step 18. Hence, our assumption is true for $|S| = 1$.

Let our assumption is true for $|S|<k$, if we can show it for $|S|=k$, then by induction, we know our assumption is true for sets of all sizes.

Let $|S|=k$, then when we break this set into smaller $S_0, S_1,...,S_{p-1}$ (Steps 21-22), either all element belong in a single $S_i$ or get distributed into multiple sets. We handle the two case separately.

\paragraph{Case 1:} Let $S$ breaks into smaller $S_0, S_1,...,S_{p-1}$.

In this case, we know all the $S_0, S_1,...,S_{p-1}$ have size less that $k$. Hence, the sizes of $(S_x-x)/p$ is also less than $k$ for all $x \in \{0,1,...,p-1\}$. Hence, we get the correct $p$-values for all elements when we call $Find\_P\_Ordering((S_x-x)/p)$ in step 24. 

From Theorem \ref{thm:p_value_translate_theorem}, we know that $v_p(S_i-i,k)=v_p((S_i-i)/p,k)+k$, hence we add $k$ to the $p$-values of all elements of the output(step 27). We know that $v_p(S_i,k)=v_p(S_i-i,k)$, hence, the $p$-values of each element are correct at the end of step 27.

Next, we show that \merge{ } preserves the $p$-values, we're done, since we know that \merge{ } doesn't update the $p$-values of any of the elements. Let an element $q$ is added at the $j^{th}$ position in the $p$-ordering output by \merge{ }. Let all the elements before this element are in the set $X$. Then, we know that the $p$-value of this element is $w_p\left(\displaystyle \prod_{a_i \in X} (q-a_i)\right)$. By Observation~\ref{thm:merge}, we know that this is equal to $w_p\left(\displaystyle \prod_{a_i \in X_{q\pmod{p}}} (q-a_i)\right)$. Since, merge doesn't re-order the $p$-orderings on any input $S_x$ while merging, we know that this is exactly the $p$-value of $q$ from before. Hence, \merge{ } preserves the $p$-values. 

Hence, the $p$-values at the end of \merge{ } are correct (step 28). Hence, \findpordering{S} gives the correct $p$-values.

\paragraph{Case 2:} Let all elements of $S$ go into a single $S_i$.

Since, we recursively keep calling \findpordering{\cdot} on the reduced set, we know at some point, we would reach case 1. As proven above, at this point, we would get the correct $p$-values. Hence, if we can show that given a correct $p$-values in step 24, \findpordering{\cdot} outputs the correct $p$-values, then by a recursive argument, this would output the correct $p$-values for any set of size $k$.

Let's say that all the elements of $S$ fall into some set $S_x$. We assume that \findpordering{(S_i-i)/p} outputs the correct $p$-values, then if we can prove that we get the correct $p$-values from $S$, then by the above argument, we are done.

From Theorem \ref{thm:p_value_translate_theorem}, we know that $v_p(S_i-i,k)=v_p((S_i-i)/p,k)+k$, hence we add $k$ to the $p$-values of all elements of the output (step 27). We know that $v_p(S_i,k)=v_p(S_i-i,k)$, hence, the $p$-values of each element are correct at the end of step 27.

Since, all the elements in $S$ are in just one $S_i$, \merge{ } acts as identity. Hence, the output at the end of Step 28 has the correct $p$-values. Hence, \findpordering{\cdot} gives the correct $p$-values for sets of size $k$.

Hence, by induction, \findpordering{\cdot} outputs the correct $p$-values on any subset of integers.  \QEDA

%% file: appendix-time-complexity-algo1.tex
\subsection{Time complexity of Algorithm~\ref{alg:MYALG}}
\label{apx:algo1-time}
\begin{theorem}
\label{thm:time-algo-1}
Given a set $S \subset \Z$ of size $n$ and a prime $p$, such that for all elements $a\in S$, $a < p^k$ for some $k$, Algorithm~\ref{alg:MYALG} returns a $p$-ordering on $S$ in $\ot(nk\log p)$ time. 
\end{theorem}

\begin{proof}[Proof] We break the complexity analysis into $2$ parts, the time complexity for merging the subsets $S_i$'s and the time complexity due the to recursive step.
\paragraph{Time complexity of \merge{S_0,S_1,...,S_{p-1}} in Algorithm~\ref{alg:MYALG}}

Let $|S_0|+|S_1|+...+|S_{p-1}|= m$. Then, the time complexity of making the heap (Step 7) is $\ot{(\min(m,p))}$ (the size of the heap). Next, the construction of common $p$-ordering (Steps 8-14) takes $\ot{(m\log{p})}$ time, this is because extraction of an element and addition of an element are both bound by $\ot{(\log{p})}$ and the runs a total of $m$ times. Hence, the total time complexity of \merge{S_0,S_1,...,S_{p-1}} is $\ot{(\min(m,p)+m\log{p})}=\ot{(m\log{p})}$ time.
\paragraph{Time complexity of Algorithm~\ref{alg:MYALG}}
Let $|S|=n$ and $S\subset\Zpk$. Then the recursion depth of \findpordering{S} is bound by $k$. Now at each depth, all the elements are distributed into multiple heaps(of sizes $m_1,m_2,...,m_q$). Hence, the sum of sizes of all smaller sets at a given depth $\sum_{i=1}^q m_i<n$. Hence, the time to run any depth is $\sum_{i=1}^q\ot{(m_i\log{p})}=\ot{(n\log{p})}$. Hence, total time complexity for $k$ depth is $\ot{(nk\log{p})}$. \QEDA
\end{proof}

%% file: appendix-rep-root-valuation.tex
\subsection{Proof of Theorem \ref{thm:correct_valuation_array}} \label{apx:valuation-rep-root}
In this appendix we prove that the $valuations$ array from Algorithm \ref{alg:rep_alg} maintains the correct valuations.

First we initialize the valuations array to zero, which implies that when we have our $p$-ordering as a null set $\phi$ and add the first element to it, we can select any number according to definition \ref{def:$p$-ordering}. 

Suppose we have generated a $p$-ordering upto length $\Tilde{n}$ with $i_1, i_2 \dots i_{|S|}$ being the number of elements from each representative root in $S$. Now if we add another element to this $p$-ordering, from say the $j^{th}$ representative root, the $p$-value contributed corresponding to each of the representative roots apart from the $j^{th}$ one will be $v_p(\beta_t - \beta_j)$ where $t \neq j$, according to Observation~\ref{thm:rep_root_interaction}. Also since we have $i_t$ many elements from each of $t^{th}$ representative root, the contribution to $p$-value will be $i_tv_p(\beta_t-\beta_j)$. Next, we find the $p$-value contributed due to the same representative root. 

Notice that, from Observation~\ref{thm:merge_rep_root} we will have the elements of the $j^{th}$ representative root as a $p$-ordering as well on $\beta_j + p^{k_j}*$, of length $i_j$. Now by Theorem \ref{thm:translate_theorem}, we will have this $p$-ordering on $\beta_j+p^{k_j}*$ as $\{\beta_j, \beta_j + p^{k_j}, \beta_j+p^{k_j}2, \dots \beta_j + p^{k_j}(i_j - 1)\}$. When we add another element to this the $p$-value contributed due to $j^{th}$ representative root will be $k_jv_p(i_j!)$.

Summing them the total $p$-value at each step, considering the next element to be added being from $j^{th}$ representative root is $\sum_{t \in [|S|]; t \neq j} i_tv_p(\beta_t - \beta_j) + k_jv_p(i_j!)$. We choose $j$ such that this expression is minimum in our algorithm.

Now, we want to show that $valuations[j] = \sum_{t \in [|S|]; t \neq j} i_tv_p(\beta_t - \beta_j) + k_jv_p(i_j!)$. We do this inductively. First we already have $0$ stored in each entry of $valuations$. Let, we have obtained a $p$-ordering upto length $\Tilde{n}$ with the respective indices as $i_1, i_2 \dots i_|S|$ with the $p$-value corresponding to addition of next element from $j^{th}$ representative root correctly stored in $valuations[j]$. Next, when we add an element from say the $t^{th}$ representative root ($t = min\_index$) we need to change the $valuations$ accordingly.

When we add this element we increase $i_t$ by one $(i_t' = i_t + 1)$. Now when we add another element, say $m$, (after the last element from the $t^{th}$ representative root), if $m \neq t$ then the new $p$-value will be $\sum_{l \in [|S|]; l \not \in \{t, m\}} i_lv_p(\beta_l - \beta_m)  + (i_t + 1)v_p(\beta_t -\beta_m)+ k_mv_p(i_m!)$ which is $v_p(\beta_t -\beta_m)$ more than the previous $valuations[m]$. So accordingly we add this value in the previous step (when we find $t$ as the $min\_index$ and then update in Steps 29-30). 

However if this $m$ (the next $min\_index$ after adding an element from $t^{th}$ representative root) is same as $t$, then the $p$-value will be $\sum_{l \in [|S|]; l \neq t} i_lv_p(\beta_l - \beta_t) + k_tv_p((i_t + 1)!)$ while the previous value of $valuations[t]$ was $\sum_{l \in [|S|]; l \neq t} i_lv_p(\beta_l - \beta_t) + k_tv_p(i_t!)$ and this difference $v_p((i_t+1)!) - v_p(i_t!)$ is stored in \pexpincrease{i_j}. We thereby update Steps 31-32 of Algorithm \ref{alg:rep_alg} to incorporate this change. Hence $valuations$ correctly stores the $p$-value as desired.
 \QEDA
 

%% file: appendix-time-complexity-rep-algo1.tex
\subsection{Time complexity of Algorithm~\ref{alg:rep_alg}}
\label{apx:rep-algo-time}
\begin{theorem}
\label{thm:time-rep-algo}
Given a set $S \subset \Zpk$, for a prime $p$ and an integer $k$, that can be represented in terms of $d$ representative roots, Algorithm~\ref{alg:rep_alg} finds a $p$-ordering of length $n$ for $S$ in $\ot(d^2k\log p + nk \log p + np)$ time.
\end{theorem}

\begin{proof}[Proof]Let $S$ contains $d$ representative roots of $\Zpk$ and we want to find the $p$-ordering up to length $n$, then, \correlate{S} runs a double loop, each of size $d$, and each iteration takes $\ot{(k\log{p})}$, hence, \correlate{S} takes  $\ot{(d^2k\log{p})}$. \pexpincrease{n} runs a single loop of size $n$ where each iteration takes $\ot{(k\log{p})}$ time, hence, it takes $\ot{(nk\log{p})}$. Then main loop run a loop of size $n$, inside this loop we do $\mathcal{O}(d)$ operations on elements of size $\log{k}$, hence, it takes $\ot{(nd)}$ time. Hence, in total, our algorithm takes $\ot{(d^2k\log{p}+nk\log{p}+nd)}$ time.\QEDA
\end{proof}

%% file: appendix-struct-root-setc.tex
\section{Structure of root sets}
\begin{observation}
\label{thm:minimal-complete-tree}
Let $f(x)=\sum_{i=0}^\infty b_i\cdot x^i \in \Zpk[x]$, for $k<p$($k$ is small), be a polynomial with root-set $A$. Let $\alpha_i\equiv j \bmod{p}$ for all $i\in [k]$, be $k$ numbers such that for no $i,j$, $\alpha_i-\alpha_j\nequiv 0 \bmod{p^2}$. Let $\alpha_i \in A$, for all $i\in [k]$, then $S_j = \{s\in \Zpk\mid s \equiv j \bmod{p}\} \subseteq A$.
\end{observation}
\begin{proof}
Let, for all $i\in [k]$, $\alpha_l = j + p*\beta_l$, then since $\alpha_l$ is in the root set of $f(\cdot)$, therefore, $$f(j + p*\beta_l)=\sum_{i=0}^\infty b_l\cdot (j + p*\beta_l)^i\equiv 0 \bmod{p^k}.$$ 
Hence, $$\sum_{i=0}^{k-1} p^i\cdot\beta_l^i\cdot g_i(j)\equiv 0 \bmod{p^k},$$ where, $g_i(x) = \sum_{n=0}^\infty  \binom{n+i}{n}\cdot b_{n+i}\cdot x^n$. Writing this system of equations in the form of matrices $B\cdot X = 0 \bmod{p^k}$, we get, 
\[
\begin{bmatrix}
1 & \beta_0 & \cdots & \beta_0^{k-1} \\
1 & \beta_1 & \cdots & \beta_1^{k-1} \\
\vdots & \vdots & \ddots & \vdots \\
1 & \beta_{k-1} & \cdots & \beta_{k-1}^{k-1} \\
\end{bmatrix}
\begin{bmatrix}
g_0(j) \\
p\cdot g_1(j) \\
\vdots \\
p^{k-1}\cdot g_{k-1}(j) \\
\end{bmatrix} = 
\begin{bmatrix}
0 \\
0 \\
\vdots \\
0 \\
\end{bmatrix}
\bmod{p^k}.
\]
Here, $|det(B)|=\left\lvert\prod\limits_{i\neq j\in[k]} (\beta_i-\beta_j)\right\rvert$. Since $\beta_i-\beta_j\nequiv 0 \bmod{p}$, therefore, $det(B)\nequiv 0 \bmod{p}$. Hence, $B$ has an inverse. Multiplying by the inverse on both sides, we get,
\[
\begin{bmatrix}
g_0(j) \\
p\cdot g_1(j) \\
\vdots \\
p^{k-1}\cdot g_{k-1}(j) \\
\end{bmatrix} = 
\begin{bmatrix}
0 \\
0 \\
\vdots \\
0 \\
\end{bmatrix}
\bmod{p^k}, or, 
\]
for $i \in [k]$, $g_i(j)\equiv 0 \bmod{p^{k-i}}$. Hence, for any element $j+p\cdot\beta \in S_j$, $f(j + p*\beta)=\sum_{i=0}^{k-1} p^i\cdot\beta_l^i\cdot g_i(j)\equiv 0 \bmod{p^k}$ (since $p^i\cdot g_i(j)\equiv 0 \bmod{p^k}$). Therefore, all elements of $S_j$ are a root of $f(\cdot)$, or $S_j \subseteq A$.\QEDA
\end{proof}
\subsection{Structure of root sets in $\Zp{2}$}
\label{apx:struct-2k}
From Section~\ref{sec:struct-2k}, we know if $\alpha=\alpha_0+\alpha_1\cdot p\in\Zp{2}$ be a root of some $f(x)$ in $\Zp{2}$.
Then $$f(\alpha_0)+p\cdot\alpha_1\cdot f'(\alpha_0)=0\bmod{p^2}.$$
Fixing $\alpha_0$ to some $j$, we start looking at structures.
\subsubsection{Case 1: root set contains atleast two roots}
Let, our root set $R_j$ contains two distinct roots, say $j+\alpha_1^0\cdot p$ and $j+\alpha_1^1\cdot p$. Then, $$f(j)+p\cdot\alpha_1^0\cdot f'(j)=0\bmod{p^2},$$ and $$f(j)+p\cdot\alpha_1^1\cdot f'(j)=0\bmod{p^2}.$$ 
Solving the above 2 equations, we get $$f(j)=0\bmod{p^2},$$ and $$f'(j)=0\bmod{p}.$$ Hence,  any $j+\tilde{\alpha}_1\cdot p$, for $\tilde{\alpha}_1\in [p]$, is a root of the polynomial, or $R_j=j+p\cdot\ast.$ 

Since, there's no free variable in $R_j$, we just have $1$ root-set of this structure.
\subsubsection{Case 2: root set contains one root}
Let, our root set $R_j$ contains just one root, say $j+\alpha_1\cdot p$. Then, $$f(j)+p\cdot\alpha_1\cdot f'(j)=0\bmod{p^2}.$$ One can easily see that no new roots seep in at this point and a root set of this form is possible\footnote{Namely $f(x)=x-(j+\alpha_1\cdot p)$.}. Hence, $R_j=j+p\cdot\alpha_1,$ for $\alpha_1\in [p]$.

Since, $\alpha_1\in [p]$, we just have $p$ root-sets of this structure.
\subsubsection{Case 3: root set is empty}
Let our root set is empty.\footnote{Namely $f(x)=a$, where $a\neq 0$}. Hence, $R_j=\emptyset.$ 

Since, there's no free variable in $R_j$, we just have $1$ root-set of this structure.

Therefore, $N_{p_2}=p+2$. Hence, 
\[   
R_j = 
     \begin{cases}
       j+p\cdot\ast\text{,} \\
       j+p\cdot\alpha\text{, for } \alpha\in [p]\text{,} \\ 
       \emptyset\text{.}
     \end{cases}
\]

\subsection{Structure of root sets in $\Zp{3}$}
\label{apx:struct-3k}
From Section~\ref{sec:struct-3k}, we know if $\alpha=\alpha_0+\alpha_1\cdot p+\alpha_2\cdot p^2\in\Zp{3}$ be a root of some $f(x)$ in $\Zp{3}$.
Then, $$f(\alpha_0)+p\cdot\alpha_1\cdot f'(\alpha_0)+\left((\alpha_1)^2\cdot\frac{f''(\alpha_0)}{2}+\alpha_2\cdot f'(\alpha_0)\right)\cdot p^2=0\bmod{p^3}.$$
Fixing $\alpha_0$ to some $j$, we start looking at structures.
\subsubsection{Case 1: root set contains atleast three roots different at $p^1$}
Let, our root set $R_j$ contains three roots, say $j+\alpha_1^0\cdot p+\alpha_2^0\cdot p^2$, $j+\alpha_1^1\cdot p+\alpha_2^1\cdot p^2$ and $j+\alpha_1^2\cdot p+\alpha_2^2\cdot p^2$ for $\alpha_1^0\neq\alpha_1^1\neq\alpha_1^2$. Then, substituting the value and solving the 3 equations, we get $$f(j)=0\bmod{p^3},$$ $$f'(j)=0\bmod{p^2},$$ and $$f''(j)=0\bmod{p}.$$ Hence, any $j+\tilde{\alpha}_1\cdot p+\tilde{\alpha}_2\cdot p^2$, for $\tilde{\alpha}_1,\tilde{\alpha}_2\in [p]$, is a root of the polynomial, or $R_j=j+p\cdot\ast.$ 

Since, there's no free variable in $R_j$, we just have $1$ root-set of this structure.
\subsubsection{Case 2: root set contains two roots different at $p^1$}
Let, our root set $R_j$ contains three roots, say $j+\alpha_1^0\cdot p+\alpha_2^0\cdot p^2$ and $j+\alpha_1^1\cdot p+\alpha_2^1\cdot p^2$ for $\alpha_1^0\neq\alpha_1^1$. Then, substituting the value and solving the 2 equations, we get $$f(j)=0\bmod{p^2},$$ and $$f'(j)=0\bmod{p}.$$ Hence, any $j+\alpha_1^0\cdot p+\tilde{\alpha}_2\cdot p^2$, for $\tilde{\alpha}_2\in [p]$, and $j+\alpha_1^1\cdot p+\tilde{\alpha}_2\cdot p^2$, for $\tilde{\alpha}_2\in [p]$, is a root of the polynomial, or $R_j=\left(j+p\cdot\alpha_1^0+p^2\cdot\ast\right)\cup\left(j+p\cdot\alpha_1^1+p^2\cdot\ast\right),$ for $\alpha_1^0\neq\alpha_1^1$ and $\alpha_1^0,\alpha_1^1\in\{0,1,...p-1\}.$

Since, $\alpha_1^0\neq\alpha_1^1$ and $\alpha_1^0,\alpha_1^1\in\{0,1,...p-1\}$, we just have $\frac{p\cdot(p-1)}{2}$ root-sets of this structure.
\subsubsection{Case 3: root set contains two roots same at $p^1$}
Let, our root set $R_j$ contains three roots, say $j+\alpha_1\cdot p+\alpha_2^0\cdot p^2$ and $j+\alpha_1\cdot p+\alpha_2^1\cdot p^2$ for $\alpha_2^0\neq\alpha_2^1$. Then, substituting the value and solving the 2 equations, we get $$f(j)=0\bmod{p^2},$$ and $$f'(j)=0\bmod{p}.$$ Hence, any $j+\alpha_1\cdot p+\tilde{\alpha}_2\cdot p^2$, for $\tilde{\alpha}_2\in [p]$, is a root of the polynomial, or $R_j=j+p\cdot\alpha_1+p^2\cdot\ast,$ for $\alpha_1,\alpha_1^1\in\{0,1,...p-1\}.$

Since, $\alpha_1\in\{0,1,...p-1\}$, we just have $p$ root-sets of this structure.
\subsubsection{Case 4: root set contains one root}
Similar to Appendix~\ref{apx:struct-2k}, we can have just one root $\alpha=j+\alpha_1\cdot p+\alpha_2\cdot p^2$ as a root of $f(x)$ and no new roots seep in. Hence, $R_j=j+p\cdot\alpha_1+p^2\cdot\alpha_2,$ for $\alpha_1,\alpha_2\in [p]$.

Since, $\alpha_1,\alpha_2\in [p]$, we just have $p^2$ root-set of this structure.
\subsubsection{Case 5: root set is empty}
Similar to Appendix~\ref{apx:struct-2k}, our root set can be empty. Hence, $R_j=\emptyset.$ 

Since, there's no free variable in $R_j$, we just have $1$ root-set of this structure.

Therefore, $N_{p_3}=\frac{3p^2+p+4}{2}$. Hence, 
\[   
R_j = 
     \begin{cases}
       j+p\cdot\ast\text{,} \\
       (j+p\cdot\alpha_1+p^2\ast)\cup(j+p\cdot\alpha_2+p^2\ast)\text{, for } \alpha_1\neq\alpha_2\in [p]\text{,} \\
       j+p\cdot\alpha+p^2\ast\text{, for } \alpha\in [p]\text{,} \\
       j+p\cdot\alpha_1+p^2\cdot\alpha_2\text{, for } \alpha_1,\alpha_2\in [p]\text{,} \\ 
       \emptyset\text{.}
     \end{cases}
\]

\subsection{Structure of root sets in $\Zp{4}$}
\label{apx:struct-4k}
From Section~\ref{sec:struct-4k}, we know if $\alpha=\alpha_0+\alpha_1\cdot p+\alpha_2\cdot p^2+\alpha_3\cdot p^3\in\Zp{4}$ be a root of some $f(x)$ in $\Zp{4}$.
Then, \begin{multline*}
    f(\alpha_0)+p\cdot\alpha_1\cdot f'(\alpha_0)+\left((\alpha_1)^2\cdot\frac{f''(\alpha_0)}{2}+\alpha_2\cdot f'(\alpha_0)\right)\cdot p^2\\
    +\left((\alpha_1)^3\cdot\frac{f'''(\alpha_0)}{6}+2\cdot\alpha_1\cdot\alpha_2\cdot f''(\alpha_0)+\alpha_3\cdot f'(\alpha_0)\right)\cdot p^3=0\bmod{p^4}.
\end{multline*}

Fixing $\alpha_0$ to some $j$, we start looking at structures.
\subsubsection{Case 1: root set contains atleast four roots different at $p^1$}
Let, our root set $R_j$ contains four roots different at $p^1$. Then, substituting the value and solving the 4 equations, we get $$f(j)=0\bmod{p^4},$$ $$f'(j)=0\bmod{p^3},$$ $$f''(j)=0\bmod{p^2},$$ and $$f'''(j)=0\bmod{p}.$$ Hence, $R_j=j+p\cdot\ast.$ 

Since, there's no free variable in $R_j$, we just have $1$ root-set of this structure.
\subsubsection{Case 2: root set contains three roots different at $p^1$}
Let, our root set $R_j$ contains three roots different at $p^1$. Then, substituting the value and solving the 3 equations, we get $$f(j)=0\bmod{p^3},$$ $$f'(j)=0\bmod{p^2},$$ and $$f''(j)=0\bmod{p}.$$ Hence, $$R_j = (j+p\cdot\alpha_1+p^2\ast)\cup(j+p\cdot\alpha_2+p^2\ast)\cup(j+p\cdot\alpha_3+p^2\ast)$$ Since, $\alpha_1\neq\alpha_2\neq\alpha_3\in [p]$, we have $\frac{p(p-1)(p-2)}{6}$ such root sets.
\subsubsection{Case 3: root set contains three roots of which 2 are different at $p^1$ and 2 are different at $p^2$}
Similar to last case, we get $$f(j)=0\bmod{p^3},$$ $$f'(j)=0\bmod{p^2},$$ and $$f''(j)=0\bmod{p}.$$ Hence,
$$R_j = (j+p\cdot\alpha_1+p^2\ast)\cup(j+p\cdot\alpha_2+p^2\ast)$$ Since, $\alpha_1\neq\alpha_2\in [p]$, we have $\frac{p(p-1)}{2}$ such root sets.
\subsubsection{Case 4: root set contains two roots different at $p^2$}
Similar to last case, we get $$f(j)=0\bmod{p^3},$$ $$f'(j)=0\bmod{p^2},$$ and $$f''(j)=0\bmod{p}.$$ Hence, $$R_j = (j+p\cdot\alpha_1+p^2\ast)$$ Since, $\alpha_1\in [p]$, we have $p$ such root sets.
\subsubsection{Case 5: root set contains two roots different at $p^1$}
Let, our root set $R_j$ contains two roots different at $p^1$. Then, substituting the value and solving the 2 equations, we get $$f(j)=0\bmod{p^2},$$ and $$f'(j)=0\bmod{p}.$$ Hence,$$R_j = (j+p\cdot\alpha_1+p^2\cdot\beta_1+p^3\ast)\cup(j+p\cdot\alpha_2+p^2\cdot\beta_2+p^3\ast)$$ Since, $\alpha_1,\alpha_2,\beta_1,\beta_2\in [p]$ and $\alpha_1\neq\alpha_2$, we have $\frac{p^3(p-1)}{2}$ such root sets.
\subsubsection{Case 6: root set contains two roots different at $p^3$}
Let, our root set $R_j$ contains two roots different at $p^3$. Then, substituting the value and solving the 2 equations, we get $$f(j)=0\bmod{p^2},$$ and $$f'(j)=0\bmod{p}.$$ Hence, $$R_j = (j+p\cdot\alpha_1+p^2\cdot\alpha_2+p^3\cdot\ast)$$ Since, $\alpha_1,\alpha_2\in [p]$, we have $p^2$ such root sets.
\subsubsection{Case 7: root set is a single element}
Similar to the Appendix~\ref{apx:struct-3k}, $$R_j = (j+p\cdot\alpha_1+p^2\cdot\alpha_2+p^3\cdot\alpha_3)$$ Since, $\alpha_1,\alpha_2,\alpha_3\in [p]$, we have $p^3$ such root sets.
\subsubsection{Case 8: root set is a empty}
$$R_j = \emptyset$$ We have 1 such root set.

Therefore, $N_{p_4}=\frac{3p^4+4p^3+6p^2+5p+12}{6}$. Hence, 
\[   
R_j = 
     \begin{cases}
       j+p\cdot\ast\text{,} \\
       (j+p\cdot\alpha_1+p^2\ast)\cup(j+p\cdot\alpha_2+p^2\ast)\cup(j+p\cdot\alpha_3+p^2\ast)\text{, for } \alpha_1\neq\alpha_2\neq\alpha_3\in [p]\text{,} \\
       (j+p\cdot\alpha_1+p^2\ast)\cup(j+p\cdot\alpha_2+p^2\ast)\text{, for } \alpha_1\neq\alpha_2\in [p]\text{,} \\
       j+p\cdot\alpha+p^2\ast\text{, for } \alpha\in [p]\text{,} \\
       (j+p\cdot\alpha_1+p^2\cdot\beta_1+p^3\ast)\cup(j+p\cdot\alpha_2+p^2\cdot\beta_2+p^3\ast)\text{, for } \alpha_1\neq\alpha_2\,\beta_1,\beta_2\in [p]\text{,}\\
       j+p\cdot\alpha_1+p^2\cdot\alpha_2+p^3\cdot\ast\text{, for } \alpha_1,\alpha_2\in [p]\text{,} \\ 
       j+p\cdot\alpha_1+p^2\cdot\alpha_2+p^3\cdot\alpha_3\text{, for } \alpha_1,\alpha_2,\alpha_3\in [p]\text{,} \\ 
       \emptyset\text{.}
     \end{cases}
\] 